\documentclass[12pt]{amsart}
\usepackage{amsmath}
\usepackage{amsthm}
\usepackage{amsfonts}
\usepackage{amssymb}
\usepackage[all,cmtip]{xy}           %xypic macro for latex2.09
\usepackage{bm}
\usepackage{bbm}
\usepackage{bbding}
\usepackage{txfonts}
\usepackage{amscd}
\usepackage{stmaryrd}
\usepackage{xspace}
\usepackage[shortlabels]{enumitem}
\usepackage{ifpdf}
\usepackage{graphicx}
\usepackage{stmaryrd}
\usepackage[all]{xy}
\ifpdf
  \usepackage[colorlinks,final,backref=page,hyperindex]{hyperref}
\else

\fi
\usepackage[active]{srcltx}
\usepackage{tikz-cd}
\usepackage{tikz}

\usetikzlibrary{positioning,fit}
\makeatletter

\newtheorem{thm}{Theorem}[section]
\newtheorem{lem}[thm]{Lemma}

\newtheorem{pro}[thm]{Proposition}
\theoremstyle{definition}

\newtheorem{defi}[thm]{Definition}

\newcommand{\nc}{\newcommand}
\newcommand{\delete}[1]{}

\nc{\mlabel}[1]{\label{#1}}  % Use this to suppress names
\nc{\mcite}[1]{\cite{#1}}  % Use this to suppress names
\nc{\mref}[1]{\ref{#1}}  % Use this to suppress names
\nc{\meqref}[1]{\eqref{#1}}  % Use this to suppress names
\nc{\mbibitem}[1]{\bibitem{#1}} % Use this to show number
\delete{% Use the next two lines to show names
\nc{\mlabel}[1]{\label{#1}{\hfill \hspace{1cm}{\bf{{\ }\hfill(#1)}}}}
\nc{\mcite}[1]{\cite{#1}{{\bf{{\ }(#1)}}}}  % Use this lines to show names
\nc{\mref}[1]{\ref{#1}{{\bf{{\ }(#1)}}}}  % Use this lines to show names
\nc{\meqref}[1]{\eqref{#1}{{\bf{{\ }(#1)}}}}  % Use this lines to show names
\nc{\mbibitem}[1]{\bibitem[\bf #1]{#1}} % Use this to show name
}

\DeclareMathOperator{\im}{Im}

\newcommand {\emptycomment}[1]{}
\delete{

}
\nc{\oprn}{\theta}

\delete{
\setlength{\baselineskip}{1.8\baselineskip}

\newcommand{\emptycomment}[1]{}

}

\nc{\calo}{\mathcal{O}}
\nc{\oop}{$\mathcal{O}$-operator\xspace}
\nc{\oops}{$\mathcal{O}$-operators\xspace}
\nc{\mrho}{{\bm{\varrho}}}
\nc{\bfk}{\mathbf{K}}
\nc{\invlim}{\displaystyle{\lim_{\longleftarrow}}\,}
\nc{\ot}{\otimes}
\newcommand{\ahd}{\vartriangleleft}
\nc{\eval}[1]{\Big|_{#1}}

\newcommand{\be }{\begin{equation}}
\newcommand{\ee }{\end{equation}}

\nc{\RR}{\mathbb{R}}

\nc{\hC}{\mathcal{C}}

\newcommand{\frkg}{\mathfrak g}

\newcommand{\br}[1]{   [ \cdot,    \cdot  ]   }
\newcommand{\id}{\mathsf{id}}

\newcommand{\wtd}{\widetilde}
\nc{\CV}{\mathbf{C}}

\NewDocumentEnvironment{Thm}{O{thm} D(){} m}
  {\addtocounter{#1}{-1}%
   \expandafter\renewcommand\csname the#1\endcsname{\ref{#3}}%
   \begin{#1}[#2]}
  {\end{#1}}

\begin{document}

\title[Matched pairs of Hopf algebras and Rota-Baxter Hopf algebras]{Matched pairs of Hopf algebras and Rota-Baxter Hopf algebras}

\author{Shukun Wang}
\address{School of Mathematics and Big Data, Anhui University of Science and Technology, Huainan 232001, China; Anhui Province Engineering Laboratory for Big Data Analysis and Early Warning Technology of Coal Mine Safety, Huainan 232001, China}
\email{2024093@aust.edu.cn}

%\date{\today}

\begin{abstract}
In this paper, we first study Rota-Baxter Hopf algebras of weight $-1$ and construct a matched pair of Hopf algebras on every Rota-Baxter Hopf algebra of weight $-1$. Then we propose the notion of projection homomorphism pairs on a matched pair of Hopf algebras, and show that every projection homomorphism pair $(C,\wtd{C})$ induces a Rota-Baxter Hopf algebra. Conversely, we prove that 
the matched pair of Hopf algebras on a Rota-Baxter Hopf algebra of weight $-1$ $(H,B)$ gives rise to a projection homomorphism pair $(C,\wtd{C})$. Furthermore, we study the Rota-Baxter Hopf algebra structure on $\im C$ that is Rota-Baxter isomorphic to $(H,B)$, and investigate the relationship between the Rota-Baxter Hopf algebra structure on $\im \wtd{C}$ and the descendent Rota-Baxter Hopf algebra $(H_{B},B)$.

\end{abstract}

%\subjclass[2020]{17B38,  17B62,   22E60, 53D17}

\keywords{matched pairs of Hopf algebras, Rota-Baxter Hopf algebras, projection homomorphism pairs. \\
\quad  2020 \emph{Mathematics Subject Classification.} 16T05.}

\maketitle

%\tableofcontents

\section{Introduction}

Rota-Baxter operators on Lie algebras and associative algebras play an important role in various fields, such as combinatorics, as seen in the works of F. V. Atkinson and P. Cartier \cite{AK,PC}. Given a Lie algebra $\frkg$, a Rota-Baxter operator of weight $\lambda$ on $\frkg$ is a linear map $B: \frkg \to \frkg$ satisfying
$$[B(x),B(y)] = B([B(x),y] + [x,B(y)] + \lambda [x,y])$$
for any $x,y \in \frkg$. The pair $(\frkg,B)$ is called a Rota-Baxter Lie algebra of weight $\lambda$.

Rota-Baxter operators are closely related to the classical Yang-Baxter equation and integrable systems \cite{BE,STS,S2}, algebraic renormalisation in quantum field theory \cite{AC}, and double Lie algebras \cite{MG1}. Motivated by the Global Factorisation Theorem for Lie groups, the notions of Rota-Baxter operators on groups and Rota-Baxter groups were introduced in \cite{LG1}. It was shown that the differentiation of a Rota-Baxter group gives rise to a Rota-Baxter Lie algebra of weight $1$. Relationships between Rota-Baxter Lie algebras and matched pairs of Lie algebras, as well as between Rota-Baxter groups and matched pairs of groups, were established in \cite{HL}. This line of research has been further advanced in \cite{W}.

Recently, generalising the notions of Rota-Baxter operators on Lie algebras and groups, Rota-Baxter Hopf algebras on cocommutative Hopf algebras were introduced in \cite{MG}. A linear map on a cocommutative Hopf algebra $H$ is called a Rota-Baxter operator of weight $1$ if it is a coalgebra map satisfying
$$B(x)B(y)=B(x_1\,B(x_2)\,y\,S(B(x_3)))$$
for any $x,y\in H$, and the pair $(H,B)$ is called a Rota-Baxter Hopf algebra. In recent years, the theory of Rota-Baxter Hopf algebras has developed rapidly; see \cite{SC,ZH}.

Matched pairs of Hopf algebras were first proposed by W. M. Singer \cite{Sin}. Early works on matched pairs of Hopf algebras include those of M. Takeuchi \cite{MT} and S. Majid \cite{MJ90,MJ}. For recent developments, we refer to \cite{DF,YL}. In 2023, it was shown in \cite{YL23} that every Rota-Baxter Hopf algebra induces a matched pair of Hopf algebras and gives rise to a solution to the Yang-Baxter equation.

The aim of this work is to investigate the relationship between matched pairs of Hopf algebras and Rota-Baxter Hopf algebras. We first introduce Rota-Baxter Hopf algebras of weight $-1$, generalising some results in \cite{MG}. Based on the results of \cite{HL} and \cite{YL23}, we construct a matched pair of Hopf algebras from each Rota-Baxter Hopf algebra of weight $-1$. We then propose the notion of projection homomorphism pairs on matched pairs of Hopf algebras. As one of our main results, we prove that the matched pair of Hopf algebras associated to every Rota-Baxter Hopf algebra of weight $-1$ gives rise to a projection homomorphism pair. Furthermore, we study the Rota-Baxter Hopf algebras induced by such projection homomorphism pairs.

This paper is organised as follows. In Section~\ref{2}, we recall the basic notions of Rota-Baxter Hopf algebras and matched pairs of Hopf algebras, and study the algebraic properties of Rota-Baxter Hopf algebras of weight $-1$. We then show that every Rota-Baxter Hopf algebra $(H,B)$ of weight $-1$ gives rise to a matched pair of Hopf algebras. In Section~\ref{3}, we introduce projection homomorphism pairs on matched pairs of Hopf algebras. We prove that every projection homomorphism pair on a matched pair of Hopf algebras induces a Rota-Baxter Hopf algebra of weight $-1$, and that the matched pair of Hopf algebras induced by a Rota-Baxter Hopf algebra $(H,B)$ of weight $-1$ yields a projection homomorphism pair $(C,\widetilde{C})$. Furthermore, we show that there exists a Rota-Baxter Hopf algebra structure on $\operatorname{Im} C$ which is Rota-Baxter isomorphic to $(H,B)$. Moreover, we study the Baxter Hopf algebra structure on $\operatorname{Im} \widetilde{C}$, and investigate its relationship with the descendant Rota-Baxter Hopf algebra $(H_B,B)$.

\vspace{2mm}
\noindent
{\bf Convention. }

Throughout this paper, the base field is a field $\mathbb{K}$ of characteristic $0.$ Unless otherwise stated, all vector spaces, algebras, and tensor products are taken over $\mathbb{K}.$

For a unital algebra $(A,m,u),$ we denote the multiplication by $m$ and the unit by $u.$

For a coalgebra $(C,\Delta,\epsilon),$ we denote the comultiplication by $\Delta$ and the counit by $\epsilon$. Using the Sweedler notation, we write $\Delta$ simply as $$\Delta(x)=x_1\otimes x_2.$$ Furthermore, for any $n\ge 1$, by the coassociativity of $\Delta,$ we can write
$$\Delta_{n}(x)=(\Delta(x)\otimes \id^{\otimes(n-1)})\cdots (\Delta(x)\otimes \id)\Delta(x).$$

By convention, we denote a Hopf algebra by $H=(H,\cdot,1,\Delta,\epsilon,S).$

 \section{The matched pair of Hopf algebras on a Rota-Baxter Hopf algebra of weight -1}\label{2}
In this section, we first recall some basic notations and preliminaries on matched pairs of Hopf algebras and Rota-Baxter Hopf algebras. Then we investigate the algebraic properties of Rota-Baxter Hopf algebras of weight $-1$. Finally, we prove that every Rota-Baxter Hopf algebra of weight $-1$ gives rise to a matched pair of Hopf algebras.

\subsection{Preliminaries on Matched Pairs of Hopf Algebras}

First, let us recall some basic notions related to Hopf algebras. For further details on Hopf algebras, we refer to \cite{SD}.

Let $(C, \Delta, \epsilon)$ be a coalgebra. A linear map $f: C \to C$ is called a coalgebra map if it satisfies the following conditions:
\begin{itemize}
    \item[(a)] $f(\Delta(x)) = \Delta(f(x)), \quad \forall x \in C$;
    \item[(b)] $\epsilon(f(x)) = \epsilon(x), \quad \forall x \in C$.
\end{itemize}

Let $H$ be a Hopf algebra. $H$ is called cocommutative if:
$$ x_1 \otimes x_2 = x_2 \otimes x_1, \quad \forall x \in H. $$

If $H$ is cocommutative, it follows from \cite{SD} that the antipode $S$ of $H$ is a coalgebra map, and that $S^2 = \id$.

Next, let us recall the definition of \textit{matched pairs of Hopf algebras} given in \cite{MJ}.

\begin{defi}
Let $H$ and $H'$ be Hopf algebras. 
Let $\rhd: H \otimes H' \to H'$ be a left action of $H$ on $H'$, and let $\lhd: H \otimes H' \to H$ be a right action of $H'$ on $H$, such that $H'$ is a left $H$-module coalgebra and $H$ is a right $H'$-module coalgebra with respect to the actions $\rhd$ and $\lhd$ respectively. The quadruple $(H, H', \rhd, \lhd)$ is called a \textbf{matched pair of Hopf algebras} if it satisfies:
\begin{align}
x \rhd 1_{H'} &= \epsilon_H(x)\,1_{H'}, \label{M1}\\
1_H \lhd a &= \epsilon_{H'}(a)\,1_H, \label{M2}\\
x \rhd (ab) &= (x_1 \rhd a_1)\,((x_2 \lhd a_2)\rhd b), \label{M3}\\
(xy) \lhd a &= (x \lhd (y_1 \rhd a_1))\,(y_2 \lhd a_2), \label{M4}
\end{align}
for all $x, y \in H$ and $a, b \in H'$. 

\end{defi}

A matched pair of Hopf algebras $(H, H', \rhd, \lhd)$ induces the \emph{double cross product} $H' \bowtie H$, 
which is a Hopf algebra structure on $H' \otimes H$ endowed with the product 
\[
(a \otimes x)(a' \otimes x') = a\,(x_1 \rhd a'_1) \otimes (x_2 \lhd a'_2)\,x', 
\quad \forall\, a, a' \in H', \ \forall\, x, x' \in H.
\]
and the usual coproduct. The antipode of $H' \bowtie H$ is given by 
\[
S(a \otimes x) = (S(x_1) \rhd S(a_1)) \otimes (S(a_2) \lhd S(x_2)), 
\quad \forall\, a \in H', \ \forall\, x \in H.
\]

\subsection{Rota-Baxter Hopf algebras of weight -1}

Motivated by the notion of \textit{Rota-Baxter operators} of weight $-1$ on groups introduced in \cite{HL} and by the study of \textit{Rota-Baxter Hopf algebras} of weight $1$ in \cite{MG}, now we define Rota-Baxter Hopf algebras of weight $-1$.

\begin{defi}
Let $H$ be a cocommutative Hopf algebra $B: H \to H$ be a linear operator. 
Then $B$ is called a \textbf{Rota-Baxter operator of weight $-1$} on $H$ if it is a coalgebra map and satisfies:
\begin{equation}\label{RBHRS}
B(x) B(y) = B\big( B(x_1)\, y\, S(B(x_2))\, x_3 \big), \quad \forall\, x, y \in H.
\end{equation}
The pair $(H, B)$ is then called a \textbf{Rota-Baxter Hopf algebra of weight $-1$}.
\end{defi}

We next introduce the notion of homomorphisms between Rota-Baxter Hopf algebras of weight $-1$, 
which we shall simply refer to as \emph{homomorphisms of Rota-Baxter Hopf algebras} throughout this paper.

\begin{defi}
Let $(H, B)$ and $(H', B')$ be Rota-Baxter Hopf algebras of weight $-1$. 
A map $f: H \to H'$ is called a \textbf{homomorphism of Rota-Baxter Hopf algebras} if $f$ is a Hopf algebra homomorphism satisfying
\[
f \circ B = B' \circ f.
\]
In this case, we say that $(H, B)$ is \textbf{Rota-Baxter homomorphic} to $(H', B')$.  
If $f$ is bijective, then $f$ is called an \textbf{isomorphism of Rota-Baxter Hopf algebras}, 
and $(H, B)$ is said to be \textbf{Rota-Baxter isomorphic to} $(H', B')$.
\end{defi}

Let $(\frkg, B)$ be a Rota-Baxter Lie algebra of weight $\lambda$. 
A well-known result states that the operator $\widetilde{B} = -\lambda\,\mathrm{id}_{\frkg} - B$ 
is also a Rota-Baxter operator of weight $\lambda$ on $\frkg$. 
As the Hopf algebra analogue, and following the argument of \cite[Proposition~1]{MG}, 
we obtain the following proposition.

\begin{pro}\label{WTB}
Let $(H, B)$ be a Rota-Baxter Hopf algebra of weight $-1$. 
Then the operator $\widetilde{B}: H \to H$ defined by 
\begin{equation}\label{HBWBD}
\widetilde{B}(x) = x_1\, B(S(x_2)), \quad \forall\, x \in H,
\end{equation}
is a Rota-Baxter operator of weight $-1$ on $H$.
\end{pro}

By the definition of $\widetilde{B}$, one can readily check that 
\begin{equation}\label{HBWBC}
B(x) = x_1\, \widetilde{B}(S(x_2)), \quad \forall\, x \in H.
\end{equation}

By \cite[Theorem~3]{MG},  every Rota-Baxter Hopf algebra $(H, B)$ 
gives rise to another Hopf algebra structure on $H$, 
which is called the \textbf{descendent Hopf algebra} of $(H, B)$. Following a similar argument, we obtain the following proposition.

\begin{pro}\label{AT}
Let $(H, B)$ be a Rota-Baxter Hopf algebra of weight $-1$. 
Then $H_{B} = (H, *_B, 1, \Delta, \epsilon, S)$ is a cocommutative Hopf algebra, 
where $*_B : H \otimes H \to H$ is defined by 
\begin{equation*}
x *_B y = B(x_1)\, y\, S(B(x_2))\, x_3, \quad \forall\, x, y \in H,
\end{equation*}
and $S_B : H \to H$ is defined by 
\begin{equation}\label{S_WB}
S_B(x) = S(B(x_1))\, S(x_2)\, B(x_3), \quad \forall\, x \in H.
\end{equation}
\end{pro}

It follows from Propositions~\ref{WTB} and~\ref{AT} that 
$H_{\widetilde{B}} = (H, *_{\widetilde{B}}, 1, \Delta, \epsilon, S)$ is also a Hopf algebra.

Finally, following the proof of \cite[Proposition~6 and Corollary~3]{MG}, we obtain the following result.

\begin{pro}\label{BH}
Let $(H, B)$ be a Rota-Baxter Hopf algebra of weight $-1$. 
Then $B$ is a homomorphism of Hopf algebras from $H_B$ to $H$. 
Furthermore, $\operatorname{Im} B$ is a Hopf subalgebra of $H$.
\end{pro}

Let $(H, B)$ be a Rota-Baxter Hopf algebra of weight $-1.$ 
Next, we show that there exists a Rota-Baxter Hopf algebra structure on $H_B$, 
called the \textbf{descendent Rota-Baxter Hopf algebra} of $(H, B).$

\begin{pro}
Let $(H, B)$ be a Rota-Baxter Hopf algebra of weight $-1.$ 
Then $(H_B, B)$ is also a Rota-Baxter Hopf algebra of weight $-1.$
\end{pro}

\begin{proof}
We check that $B$ is a Rota-Baxter operator of weight $-1$ on $H_B.$ 
It suffices to verify that \eqref{RBHRS} holds. 
For any $x, y \in H$, we have
\begin{equation*}
\begin{aligned}
B\big(B(x_1) *_B y *_B S(B(x_2)) *_B x_3\big)
&= B(B(x_1))\, B(y)\, S(B(B(x_2)))\, B(x_3) \quad (\textit{by Proposition~\ref{AT}})\\
&= B(x) *_B B(y) \quad (\textit{by \eqref{RBHRS}}).
\end{aligned}
\end{equation*}
Hence, $(H_B, B)$ is a Rota-Baxter Hopf algebra of weight $-1.$
\end{proof}

By the definition of $\widetilde{B}$, one can readily check that for any $x \in H$,
\begin{equation}\label{BTTB}
\widetilde{B}(x_1)\, S(B(S(x_2))) = x, 
\quad 
B(x_1)\, S(\widetilde{B}(S(x_2))) = x.
\end{equation}

Furthermore, we have the following proposition.

\begin{pro}\label{BBCV}
Let $(H, B)$ be a Rota-Baxter Hopf algebra of weight $-1.$ Then the following identities hold:
\begin{equation*}
\begin{aligned}
\widetilde{B}(B(x)) &= B\!\left(S\!\left(\widetilde{B}(S(x))\right)\right),\\
B(\widetilde{B}(x)) &= \widetilde{B}\!\left(S\!\left(B(S(x))\right)\right),
\end{aligned}
\end{equation*}
for all $x \in H.$
\end{pro}

\begin{proof}
For any $x \in H$, we compute
\begin{equation*}
\begin{aligned}
\widetilde{B}(B(x))
&= B(x_1)\, B(S(B(x_2)))\\
&= B\big(B(x_1)\, S(B(x_4))\, S(B(x_2))\, x_3\big) 
\quad \textit{(by \eqref{RBHRS})}\\
&= B\big(B(x_1)\, S(B(x_2))\, S(B(x_3))\, x_4\big)\\
&= B\big(\epsilon(x_1)\, S(B(x_2))\, x_3\big).
\end{aligned}
\end{equation*}
It follows that 
\begin{equation*}
\begin{aligned}
\widetilde{B}(B(x))
= B(S(B(x_1))\, x_2)
= B(S(\widetilde{B}(S(x)))) \quad \textit{(by \eqref{HBWBD})}.
\end{aligned}
\end{equation*}
Since $\widetilde{B}$ is also a Rota-Baxter operator of weight $-1$ on $H$, 
the second identity follows in the same way.
\end{proof}

Next, we investigate the relationship between the products $\ast_B$ and $\ast_{\wtd{B}}$.

\begin{pro}\label{ATWB}
Let $(H,B)$ be a Rota-Baxter Hopf algebra of weight $-1$. Then the following equality holds:
$$
S(S(x)\ast_{\wtd{B}} S(y))=x\ast_{B} y,\quad \forall\, x,y\in H.
$$
\end{pro}

\begin{proof}
For any $x,y\in H$, we have
\begin{equation*}
\begin{aligned}
S(S(x)\ast_{\wtd{B}}S(y))
&=S\!\left(
\wtd{B}(S(x_3))S(y)S(\wtd{B}(S(x_2)))S(x_1)
\right)\\
&=x_1\, \wtd{B}(S(x_2))\, y\, S(\wtd{B}(S(x_3)))\\
&=x_1 S(x_2) B(x_2) y S(B(x_3)) x_4
\quad (\text{by \eqref{HBWBD} and \eqref{HBWBC}})\\
&=\epsilon(x_1) B(x_2) y S(B(x_3)) x_4\\
&=B(x_1) y S(B(x_2)) x_3
=x\ast_{B} y.
\end{aligned}
\end{equation*}
\end{proof}

\subsection{The matched pair of Hopf algebras induced by Rota-Baxter Hopf algebras}

Let $(H,B)$ be a Rota-Baxter Hopf algebra of weight $-1$. Throughout this paper, denote
$$
\begin{aligned}
H_+&=\im B, \quad H_-=\im \wtd{B},\\
K_+&=\ker \wtd{B}, \quad K_-=\ker B.
\end{aligned}
$$
By Proposition~\ref{BH}, both $H_+$ and $H_-$ are Hopf subalgebras of $H$.

In the next theorem, we prove that every Rota-Baxter Hopf algebra of weight $-1$ induces a matched pair of Hopf algebras.
\begin{thm}
Let $(H,B)$ be a Rota-Baxter Hopf algebra of weight $-1$. 
Define the linear maps $\rhd:H_+\otimes H_-\to H_-$ by
\[
B(x)\rhd \wtd{B}(y)=\wtd{B}(B(x_1)yS(B(x_2))),\quad \forall\, x,y\in H,
\]
and $\ahd:H_+\otimes H_-\to H_+$ by
\[
B(x)\ahd \wtd{B}(y)
= S\!\left(B\!\left(S(\wtd{B}(y_1))\,S_B(x)\,\wtd{B}(y_2)\right)\right),\quad \forall\, x,y\in H.
\]
Then $(H_+,H_-,\rhd,\ahd)$ is a matched pair of Hopf algebras.
\end{thm}

\begin{proof}
We first check that $\rhd$ is well defined. 
For any $x\in H$ and $y\in K_+$, since $K_+\subseteq H_+$ is an ideal, we have 
$B(x_1)yS(B(x_2))\in K_+$. 
It follows that $\wtd{B}(B(x_1)yS(B(x_2)))=0$. 
Hence $\rhd$ is well defined. The same argument applies to $\ahd$.

Next, we verify that $H_-$ is a left $H_+$-module coalgebra with respect to $\rhd$. 
By the  definition of $B$, $H_-$ is an $H_+$-module under $\rhd$. 
Moreover, for any $x,y\in H$, we have
\[
\Delta(B(x)\rhd \wtd{B}(y))
=\wtd{B}(B(x_1)y_1S(B(x_2)))\otimes \wtd{B}(B(x_3)y_2S(B(x_4)))
=\Delta(\wtd{B}(B(x_1)yS(B(x_2)))).
\]
Hence $H_-$ is a left $H_+$-module coalgebra. 
Similarly, $H_+$ is a right $H_-$-module coalgebra with respect to $\ahd$.

It is straightforward to verify that conditions \eqref{M1} and \eqref{M2} hold. 
We now prove \eqref{M3} and \eqref{M4}. 
For any $x,y\in H$, we compute
\[
\begin{aligned}
B(x)\rhd \wtd{B}(y)
&=\wtd{B}(B(x_1)yS(B(x_2)))\\
&=B(x_1)y_1B(S_B(x_2))B(S(B(x_3)y_2S(B(x_4))))\quad (\text{by \eqref{HBWBD} and Proposition~\ref{BH}})\\
&=B(x_1)y_1B\!\left(S(y_2)S_B(x_2)\right).
\end{aligned}
\]
Thus
\begin{equation}\label{MM1}
B(x)\rhd \wtd{B}(y)=B(x_1)y_1B\!\left(S(y_2)S_B(x_2)\right).
\end{equation}
Similarly, we obtain
\begin{equation}\label{MM2}
B(x)\ahd \wtd{B}(y)
=S\!\left(\wtd{B}(y_1)S_B(x_1)\wtd{B}\!\left(S_B(x_2)y_2\right)\right).
\end{equation}
From \eqref{BTTB}, \eqref{MM1}, and \eqref{MM2}, it follows that
\begin{equation}\label{MM3}
(B(x_1)\rhd \wtd{B}(y_1))(B(x_2)\ahd \wtd{B}(y_2))=B(x)\wtd{B}(y).
\end{equation}

Finally, for all $x,y,z\in H$, a direct computation using \eqref{RBHRS} and \eqref{MM3} shows that
\[
(B(x_1)\rhd \wtd{B}(y_1))\!\left((B(x_2)\ahd \wtd{B}(y_2))\rhd \wtd{B}(z)\right)
= B(x)\rhd \wtd{B}(y\ast_{\wtd{B}} z)
= B(x)\rhd (\wtd{B}(y)\wtd{B}(z)),
\]
which proves \eqref{M3}. 
The proof of \eqref{M4} is similar.
\end{proof}

The matched pair of Hopf algebras $(H_+,H_-,\rhd,\ahd)$ given in the above theorem is called \textbf{the matched pair of Hopf algebras on $(H,B)$}.

The next Proposition will be used in the proof of Theorem \ref{MT}.
\begin{pro}\label{MRBE}
Let $(H,B)$ be a Rota-Baxter Hopf algebra of weight $-1$ and $(H_+,H_-,\rhd,\ahd)$ be the matched pair of Hopf algebras on $(H,B)$. Then the following equality holds:
\begin{equation*}
\left(\wtd{B}(x_1),S\circ B\circ S(x_2)  \right)\left( \wtd{B}(y_1), S\circ B\circ S(y_2) \right)=\left( \wtd{B}(x_1y_1),S\circ B\circ S(x_2y_2) \right).
\end{equation*}
on $H_-\bowtie H_+$
\end{pro}
\begin{proof}
For any $x,y\in H,$ we have
\begin{equation*}
\begin{aligned}
&\left(\wtd{B}(x_1),S\circ B\circ S(x_2)  \right)\left(\wtd{B}(y_1),S\circ B\circ S(y_2) \right)\\
=&\left(\wtd{B}(x_1)\left(S\circ B\circ S(x_2)\rhd \wtd{B}(y_1)\right),\left(S\circ B\circ S(x_3) \ahd \wtd{B}(y_2)\right)S\circ B\circ S(y_3)\right)\\
=&\left(\wtd{B}(x_1)\wtd{B}\left(S\circ B\circ S(x_2)y_1B\circ S(x_3)  \right),  S\circ B\left(S\circ \wtd{B}(y_2) S(x_4)\wtd{B}(y_3) \right) S\circ B\circ S(y_4)\right)\ (\textit{by Proposition \ref{AT}})\\
=&\left(\wtd{B}\left(\wtd{B}(x_1)S\circ B\circ S(x_4)y_1B\circ S(x_5)S\circ \wtd{B}(x_2)x_3\right),
S\circ(B\left(B\circ S(y_4)S\circ\wtd{B}(y_2)S(x_6)\wtd{B}(y_3)S\circ B\circ S(y_5)S(y_6)        \right)
\right)\\
&(\textit{by \eqref{RBHRS}}).
\end{aligned}
\end{equation*}
It follows that 

\begin{equation*}
\begin{aligned}
&\left(\wtd{B}(x_1),S\circ B\circ S(x_2)  \right)\left(\wtd{B}(y_1),S\circ B\circ S(y_2) \right)\\
=&\left(\wtd{B}\left(\wtd{B}(x_1)S\circ B\circ S(x_2)y_1B\circ S(x_3)S\circ \wtd{B}(x_4)x_5\right),
S\circ B\left(B\circ S(y_2)S\circ\wtd{B}(y_3)S(x_6)\wtd{B}(y_4)S\circ B\circ S(y_5)S(y_6)\right)
\right)\\
=&\left(\wtd{B}\left(x_1y_1 S(x_2)x_3\right),
S\circ B\left(S(y_2)S(x_3)y_4S(y_5)\right)\right)\  (\textit{by \eqref{BTTB}})\\
=&\left(\wtd{B}\left(x_1y_1 \epsilon(x_2)\right),
S\circ B\left(S(y_2)S(x_3)\epsilon(y_3)\right)\right)\\
=&\left(\wtd{B}\left(x_1y_1 \right),
S\circ B\left(S(x_2y_2)\right)\right).
\end{aligned}
\end{equation*}

\end{proof}

\section{Rota-Baxter Hopf algebras and projection homomorphism pairs }\label{3}
In this section, we first introduce the notion of projection homomorphism pairs on matched pairs of Hopf algebras. Then we show that such pairs naturally give rise to Rota-Baxter Hopf algebra structures. Moreover, we prove that the matched pair of Hopf algebras on a Rota-Baxter Hopf algebra induces a projection homomorphism pair $(C,\widetilde{C})$. Finally, we investigate the Rota-Baxter Hopf algebra structures on $\operatorname{Im} C$ and $\operatorname{Im} \widetilde{C}$.                                                                                                                   

\subsection{Projection homomorphism pairs on matched pairs of Hopf algebras}

We first introduce the notion of projection homomorphism pairs on matched pairs of Hopf algebras.
\begin{defi}
Let $(H,H',\rhd,\ahd)$ be a matched pair of Hopf algebras. Let $B:H'\bowtie H\to H'\bowtie H$ and $B':H'\bowtie H\to H'\bowtie H$ be idempotent homomorphisms of Hopf algebras. The pair $(B,B'
)$ is called a \textbf{projection homomorphism pair on $(H,H',\rhd,\ahd)$} if it satisfies:
\begin{equation*}
B(x_1)B'(x_2)=x
\end{equation*}
for any $x\in H'\bowtie H$.
\end{defi}

In the next Proposition, we prove that every projection homomorphism pair on a matched pair of Hopf algebras $H'\bowtie H$ induces a Rota-Baxter operator of weight $-1$ on $H'\bowtie H$.
\begin{pro}\label{DMP}
Let $(H,H',\rhd,\ahd)$ be a matched pair of Hopf algebras. Let $(B,B')$ be a projection homomorphism pair on $(H,H',\rhd,\ahd)$. Then the following statements hold:
\begin{itemize}
\item[(a)]$B$ and $B'$ are both Rota-Baxter operators of $-1$ on $H'\bowtie H$ such that $B(B'(x))=B'(B(x))=\epsilon(x)1_{H'\bowtie H};$
\item[(b)]  $B'(x_1)B(x_2)=x,\ \forall x\in H'\bowtie H.$ 

\end{itemize}
\end{pro}
\begin{proof}
As $B$ is an idempotent Hopf homomorphism, we have 
\begin{equation*}
\begin{aligned}
B\left(B(x_1)y S(B(x_2))x_3\right)=&B(B(x_1))B(y) B(S(B(x_2)))B(x_3)\\=&B(B(x_1))B(y)S(B(x_2))B(x_3)\\
=&B(B(x_1))B(y)\epsilon(x_2)\\
=&B(x)B(y)
\end{aligned}
\end{equation*}
for any $x,y\in H'\bowtie H.$
This means $B$ and $B'$ are both Rota-Baxter of weight $-1$ on $H$. As $B(x_1)B'(x_2)=x$, we have $$B(x)=B(x_1)B'(x_2)S(B'(x_3))=x_1S(B'(x_2))=x_1B'(S(x_2)).$$  
Next, we have
\begin{equation*}
B(B'(x))=B'(x_1)B'(S(B'(x_2)))=B'(x_1)S(B'^{2}(x_2))=B'(x_1)S(B'(x_2)) =\epsilon(x)1_{H'\bowtie H}
\end{equation*}
for any $x\in H'\bowtie H.$ It is similar to show that $B'(B(x))=\epsilon(x)1_{H'\bowtie H}.$ This proves (a).
Finally, by the definition of $(B,B'),$ we obtain that $$B'(x_1)B(x_2)=x_1B(S(x_2))B(x_3)=x_1S(B(x_2))B(x_3)=x,$$ which proves (b).
\end{proof}

Next, we give an equivalent description of projection homomorphism pairs on matched pairs of Hopf algebras.
\begin{pro}\label{CMM}
Let $(H,H')$ be a matched pair of Hopf algebras. Let $B:H'\bowtie H\to H'\bowtie H$ be an idempotent homomorphism of Hopf algebras and $B':H'\bowtie H\to H'\bowtie H$ be the operator given by 
\begin{equation*}
B'((x,y))=(x_1,y_1)B\left(S((x_2,y_2))\right),\ \forall x,y\in H.
\end{equation*}
Then the following statements are equivalent:
\begin{itemize}
\item[(a)] $(B,B')$ is a projection homomorphism pair on $(H_+,H_-,\rhd,\ahd);$
\item[(b)] The image of $B$ and $B'$ are commutative, that is,
\begin{equation*}
B((x,y))B'((x',y'))=B'((x',y'))B((x,y)),\ \forall (x,y),(x',y')\in H'\bowtie H.
\end{equation*}
\end{itemize}
\end{pro}
\begin{proof}
(a)$\Rightarrow$ (b) For any $(x,y),(x',y')\in H'\bowtie H,$ by (a) of Proposition \ref{DMP}, we have
\begin{equation*}
B\left(B((x,y))B'((x',y'))\right)
=B(B((x,y)))B(B'((x',y')))=\epsilon(x')\epsilon(y')B((x,y)).
\end{equation*}
It is similar to show that 
\begin{equation*}
B'\left(B((x,y))B'((x',y'))\right)
=\epsilon(x)\epsilon(y)B'((x',y')).
\end{equation*}
It follows from (b) of Proposition \ref{DMP} that 
\begin{equation*}
\begin{aligned}
B((x,y))B'((x',y'))
=&B'\left(B((x_1,y_1))B'((x'_1,y'_1))\right)B\left(B((x_2,y_2))B'((x'_2,y'_2))\right)
=B'((x',y'))B((x,y)).
\end{aligned}
\end{equation*}
(b) $\Rightarrow$ (a)
It follows from (a) of Proposition \ref{DMP} that $B$ is a Rota-Baxter operator of weight $-1$ on $H.$ Then by Proposition \ref{WTB}, $B'$ is a Rota-Baxter operator of weight $-1$ on $H.$ By Proposition \ref{AT}, $B'$ is a coalgebra map. It remains to show that $B'$ is an idempotent homomorphism of algebras. It follows directly from the definition of $B'$ that 
$$
\begin{aligned}
B\circ B'((x,y))
&=B\!\left((x_1,y_1)\,S\!\left(B((x_2,y_2))\right)\right)\\
&=B((x_1,y_1))B(S(B((x_2,y_2))))\\
&=B((x_1,y_1))S(B^2((x_2,y_2)))\\
&=B((x_1,y_1))S(B((x_2,y_2)))\\
&=\epsilon((x,y))\,1_{H'\bowtie H}.
\end{aligned}
$$

Then we have 
$$
\begin{aligned}
B'^{2}((x,y))&=B'((x_1,y_1))B\circ S\circ B'((x_2,y_2))=B'((x_1,y_1))S\circ B\circ B'((x_2,y_2))\\&=B'((x_1,y_1))\epsilon((x_2,y_2))=B'((x,y)).
\end{aligned}
$$
This means $B'$ is idempotent. Finally we prove that $B'$ is a homomorphism of algebras. It follows from (b) that 
\begin{equation}\label{C1}
\begin{aligned}
(x,y)(x',y')=&B'((x_1,y_1))B((x_2,y_2))B'((x'_1,y'_1))B((x'_2,y'_2))\\
=&B'((x_1,y_1))B'((x'_1,y'_1))B((x_2,y_2))B((x'_2,y'_2))
.
\end{aligned}
\end{equation}
On the other hand, by the definition, we have
\begin{equation}\label{C2}
\begin{aligned}
(x,y)(x',y')=B'((x_1,y_1)(x'_1,y'_1))B((x_2,y_2)(x'_2,y'_2))=B'((x_1,y_1)(x'_1,y'_1))B((x_2,y_2))B((x'_2,y'_2)).
\end{aligned}
\end{equation}
Comparing\eqref{C1} and \eqref{C2}, we obtain that $B'$ is a homomorphism of algebras.
\end{proof}

In the next theorem, we prove that every projection homomorphism pair $(B,B')$ on a matched pair of Hopf algebras $(H,H',\rhd,\ahd)$ induces a Rota-Baxter Hopf algebra of weight $-1$ on $\im B.$

\begin{thm}\label{RBP}
Let $(H,H',\rhd,\ahd)$ be a matched pair of Hopf algebras and $(B,B')$ be a projection homomorphism pair on $(H,H',\rhd,\ahd)$. Let $K=\im B.$ 
Then the operator $C:K\to K$ defined by
\begin{equation*}
C((x,y))=B((x,\epsilon(y)1_H)),\ \forall (x,y)\in K,
\end{equation*}
is a Rota-Baxter operator of weight $-1$ on $K.$
\end{thm}
\begin{proof}
First, we have
\begin{equation*}
\begin{aligned}
\Delta C((x,y))&=\Delta(B((x,\epsilon(y)1_H)))=(B((x_1,\epsilon(y_1)1_H)))\otimes (B((x_2,\epsilon(y_2)1_H)))\\
&=C((x_1,y_1))\otimes C((x_2,y_2)).
\end{aligned}
\end{equation*}
This means $C$ is a coalgebra map.
For any $(x,y),(x',y')\in K,$ we have
\begin{equation*}
\begin{aligned}
&C\left(C((x_1,y_1))(x',y')S(C((x_2,y_2)))(x_3,y_3)\right)\\
=&C\left(B((x_1,\epsilon(y_1)1_H)) (x',y')S(B((x_2,\epsilon(y_2)1_H)))(x_3,y_3)   \right)\\
=&C\left(B((x_1,\epsilon(y_1)1_H)) B((x',y'))B(S((x_2,\epsilon(y_2)1_H)))B((x_3,y_3)) \right)\ (\textit{by the definition of $(B,B')$})\\
=&C\circ B\left((x_1,\epsilon(y_1)1_H)(x',y')S((x_2,\epsilon(y_2)1_H))(x_3,y_3) \right)\ (\textit{by the definition of $(B,B')$}).\\
\end{aligned}
\end{equation*}
It follows that 
\begin{equation*}
\begin{aligned}
&C\left(C((x_1,y_1))(x',y')S(C((x_2,y_2)))(x_3,y_3)\right)\\
=&C\circ B\left((x_1,\epsilon(y_1)1_H)(x',y')S((x_2,\epsilon(y_2)1_H))(x_3,\epsilon(y_3)1_H)(\epsilon(x_4)1_{H'}, y_4) \right)\\
=&C\circ B\left((x_1,\epsilon(y_1)1_H)(x',y')(\epsilon(x_2)1_{H'}, y_2) \right).
\end{aligned}
\end{equation*}

By Proposition \ref{DMP}, we have
\begin{equation*}
\begin{aligned}
&(x_1,\epsilon(y_1)1_H)(x',y')(\epsilon(x_2)1_{H'}, y_2)\\=&B\left((x_1,\epsilon(y_1)1_H)(x'_1,y'_1)(\epsilon(x_2)1_{H'}, y_2)\right)B'\left((x_3,\epsilon(y_3)1_H)(x'_2,y'_2)(\epsilon(x_4)1_{H'}, y_4)\right)\\=&B\left((x_1,\epsilon(y_1)1_H)B((x'_1,y'_1))(\epsilon(x_2)1_{H'}, y_2)\right)B'\left((x_3,\epsilon(y_3)1_H)\right)B'\left(B((x'_2,y'_2))\right)B'\left((\epsilon(x_4)1_{H'}, y_4)\right)\\
&(\textit{by the definition of $(B,B')$ and (a) of Proposition \ref{DMP}})\\
=&B\left((x_1,\epsilon(y_1)1_H)B(x'_1,y'_1)(\epsilon(x_2)1_{H'}, y_2)\right)B'\left((x_3,\epsilon(y_3)1_H)\right)\epsilon((x'_2,y'_2))1_{H'\bowtie H}B'\left((\epsilon(x_4)1_{H'}, y_4)\right)
\end{aligned}
\end{equation*}
It follows that
\begin{equation*}
\begin{aligned}
&(x_1,\epsilon(y_1)1_H)(x',y')(\epsilon(x_2)1_{H'}, y_2)\\
=&B\left((x_1,\epsilon(y_1)1_H)B(x'_1,y'_1)(\epsilon(x_2)1_{H'}, y'_2)\right)B'\left((x_3,\epsilon(y_3)1_H)\right)B'\left((\epsilon(x_4)1_{H'}, y_4)\right)\\
=&B\left((x_1,\epsilon(y_1)1_H)B(x',y')(\epsilon(x_2)1_{H'}, y_2)\right)B'\left((x_3,y_3)\right)\\
=&B\left((x_1,\epsilon(y_1)1_H)B(x',y')(\epsilon(x_2)1_{H'}, y_2)\right)B'\left(B((x_3,y_3))\right)\\
=&B\left((x_1,\epsilon(y_1)1_H)B(x',y')(\epsilon(x_2)1_{H'}, y_2)\right)\epsilon((x_3,y_3))1_
{H'\bowtie H} \ (\textit{by (a) of Proposition \ref{DMP}})\\
=&B\left((x_1,\epsilon(y_1)1_H)(x',y')(\epsilon(x_2)1_{H'}, y_2)\right)
\in H.
\end{aligned}
\end{equation*}

Then we have
\begin{equation*}
\begin{aligned}
&C\left(C((x_1,y_1))(x',y')S(C((x_2,y_2)))(x_3,y_3)\right)\\
=&C\left((x_1,\epsilon(y_1)1_H)(x',y')(\epsilon(x_2)1_{H'}, y_3) \right)\\
=&C\left((x_1,\epsilon(y_1)1_H)(x'_1,\epsilon(y'_1)1_H)(\epsilon(x'_2)1_{H'},y'_2)(\epsilon(x_2)1_{H'}, y_2) \right)\\
=&C\left((x_1 x'_1,\epsilon(y_1y'_1)1_H)(\epsilon(x'_2x_2)1_{H'},y'_2y_2)\right)\\
\end{aligned}
\end{equation*}
It follows that 
\begin{equation*}
\begin{aligned}
&C\left(C((x_1,y_1))(x',y')S(C((x_2,y_2)))(x_3,y_3)\right)\\
=&C\left((xx',y'y)\right)=B\left((xx',\epsilon(y'y)1_H)\right)=B\left((xx',\epsilon(y')\epsilon(y)1_H)\right)\\=&B\left((xx',\epsilon(yy')1_H)\right)=C((xx',yy'))=C((x,y))C((x',y')).
\end{aligned}
\end{equation*}
This proves that $C$ is a Rota-Baxter operator of weight $-1$ on $K$.
\end{proof}

To prove Theorem \ref{MT}, we introduce the following two lemmas.

\begin{lem}\label{GG1}
Let $(H,B)$ be a Rota-Baxter Hopf algebra of weight $-1$ and $(H_+,H_-,\rhd,\ahd)$ be the matched pair of Hopf algebras on $(H,B).$ Then the following equalities hold:
\begin{itemize}
\item[(a)]
$\wtd{B}\left(S(B( S(x_1)))\right)\left(S\left(B\left( \wtd{B}(x_2)\right)\right)\rhd \wtd{B}(S(y))\right)=\wtd{B}\left(S(y)S(B(S(x)))\right);$

\item[(b)] $\wtd{B}(S(y_1))\left(
S(B(y_2))\rhd \wtd{B}\left(S(B( S(x)))\right)
\right)=\wtd{B}\left(S(y)S(B(S(x)))\right);$
\item[(c)]
$\left(S\left(B\left(\wtd{B}(x)\right)\right)\ahd \wtd{B}(S(y_1))\right)S(B(y_2))  =S(B(y\wtd{B}(x)));$

\item[(d)] 
$\left(
S(B(y_3))\ahd \wtd{B}\left(S(B( S(x_1)))\right)
\right)S\left(B\left(\wtd{B}(x_3)\right)\right)=S(B(y\wtd{B}(x)))$
\end{itemize}
for any $x,y\in H.$
\end{lem}
\begin{proof}
For any $x,y\in H,$ we have
\begin{equation*}
\begin{aligned}
&\wtd{B}\left(S(B( S(x_1)))\right)\left(S\left(B\left( \wtd{B}(x_2)\right)\right)\rhd \wtd{B}(S(y))\right)\\
=&\wtd{B}\left(S(B( S(x_1)))\right)\wtd{B}\left(
S\left(B\left(\wtd{B}(x_2)\right)\right)S(y)B\left( \wtd{B}(x_3)\right)\right)\\
=&\wtd{B}\left(\wtd{B}\left(S(B( S(x_1)))\right)
S\left(B\left( \wtd{B}(x_4)\right)\right)S(y)B\left( \wtd{B}(x_5)\right)
S(\wtd{B}\left(S(B( S(x_2)))\right))S(B(S(x_3)) \right)\\
=&\wtd{B}\left(\wtd{B}\left(S(B( S(x_1)))\right)
S\left(B\left( \wtd{B}(x_2)\right)\right)S(y)B\left( \wtd{B}(x_3)\right)
S(\wtd{B}\left(S(B( S(x_4)))\right))S(B(S(x_5))) \right).
\end{aligned}
\end{equation*}
If follows that 
\begin{equation*}
\begin{aligned}
&\wtd{B}\left(S(B( S(x_1)))\right)\left(S\left(B\left( \wtd{B}(x_2)\right)\right)\rhd \wtd{B}(S(y))\right)\\
=&\wtd{B}\left(B(\wtd{B}(x_1))
S\left(B\left( \wtd{B}(x_2)\right)\right)S(y)B\left( \wtd{B}(x_3)\right)
S\left(B(\wtd{B}(x_4))\right )S(B(S(x_5)))\right)\\ &(\textit{by \eqref{RBHRS} and Proposition \ref{BBCV}}) \\
=&\wtd{B}\left(\wtd{B}\left(\epsilon(x_1)S(y)\epsilon(x_2)S(B(S(x_3))\right)\right)\ (\textit{by Proposition \ref{BBCV}}) \\
=&\wtd{B}\left(S(y)S(B(S(x)))\right).
\end{aligned}
\end{equation*}
This proves (a). Then we prove (b). For any $a,b\in H,$ we have
\begin{equation*}
\begin{aligned}
&\wtd{B}(S(y_1))\left(
S(B(y_2))\rhd \wtd{B}\left(S(B( S(x)))\right)
\right)\\
=&\wtd{B}(S(y_1))\wtd{B}\left(
S(B(y_2))S(B( S(x)))B(y_3)   
\right)\\
=&\wtd{B}\left(
\wtd{B}(S(y_1))S(B(y_4))S(B(S(x)))B(y_5) S(\wtd{B}(y_2))S(y_3)
\right)\\
=&\wtd{B}\left(
\wtd{B}(S(y_1))S(B(y_2))S(B(S(x)))B(y_3) S(\wtd{B}(y_4))S(y_5).
\right)\\
\end{aligned}
\end{equation*}
It follows that
\begin{equation*}
\begin{aligned}
&\wtd{B}(S(y_1))\left(
S(B(y_2))\rhd \wtd{B}\left(S(B( S(x)))\right)
\right)\\
=&\wtd{B}\left(
S(y_1)S(B(S(x)))y_2S(y_3)
\right)\ (\textit{by \eqref{HBWBD} and \eqref{HBWBC}})\\
=&\wtd{B}\left(
S(y_1)S(B(S(x)))\epsilon(y_2))
\right)\\
=&\wtd{B}\left(
S(y)S(B(S(x)))
\right).
\end{aligned}
\end{equation*}
This proves (b). The proof of (c) and (d) is similar.
\end{proof}

\begin{lem}\label{GY2}
Let $(H,B)$ be a Rota-Baxter Hopf algebra of weight $-1$ and $(H_+,H_-,\rhd,\ahd)$ be the matched pair of Hopf algebras on $(H,B).$ Then the following equality holds on $H_-\bowtie H_+:$
\begin{equation*}
\begin{aligned}
&\left(\wtd{B}\left(S(B( S(x_1)))\right),S\left(B\left( \wtd{B}(x_2)\right)\right) \right)\left(\wtd{B}(S(y_1)),S(B(y_2)) \right)\\
=&\left(\wtd{B}(S(y_1)),S(B(y_2)) \right)\left(\wtd{B}\left(S(B( S(x_1)))\right),S\left(B\left( \wtd{B}(x_2)\right)\right) \right)
\end{aligned}
\end{equation*}
for any $x,y\in H.$
\end{lem}
\begin{proof}
For any $x,y\in H,$ we have
\begin{equation*}
\begin{aligned}
&\left(\wtd{B}\left(S(B( S(x_1)))\right),S\left(B\left( \wtd{B}(x_2)\right)\right) \right)\left(\wtd{B}(S(y_1)),S(B(y_2)) \right)\\
=&\left(\wtd{B}\left(S(B( S(x_1)))\right)\left(S\left(B\left( \wtd{B}(x_2)\right)\right)\rhd \wtd{B}(S(y_1))\right),     \left(S\left(B\left(\wtd{B}(x_3)\right)\right)\ahd \wtd{B}(S(y_2))\right)S(B(y_3))       
\right)
\\
=&\left(\wtd{B}\left(S(y_1)S(B(S(x_1)))\right),S\left(B\left(y_2\wtd{B}(x_2)\right)\right)\right) \ (\textit{by Lemma \ref{GG1}}).
\end{aligned}
\end{equation*}
On the other hand, we have

\begin{equation*}
\begin{aligned}
&\left(\wtd{B}(S(y_1)),S(B(y_2)) \right)\left(\wtd{B}\left(S(B( S(x_1)))\right),S\left(B\left( \wtd{B}(x_2)\right)\right) \right)\\
=&\left(\wtd{B}(S(y_1))\left(
S(B(y_2))\rhd \wtd{B}\left(S(B( S(x_1)))\right)
\right),  
\left(
S(B(y_3))\ahd \wtd{B}\left(S(B( S(x_1)))\right)
\right)S\left(B\left(\wtd{B}(x_3)\right)\right)
\right)\\
=&\left(\wtd{B}\left(S(y_1)S(B(S(x_1)))\right),S(B\left(y_2\wtd{B}(x_2))\right)\right) \ (\textit{by Lemma \ref{GG1}}).
\end{aligned}
\end{equation*}

\end{proof}

Next, we prove that there is a projection homomorphism pair on the matched pair of Hopf algebras on a Rota-Baxter Hopf algebra of weight $-1$.
\begin{thm}\label{MT}
Let $(H,B)$ be a Rota-Baxter Hopf algebra of weight $-1$ and $(H_+,H_-,\rhd,\ahd)$ be the matched pair of Hopf algebras on $(H,B).$ Define $C:H_-\bowtie H_+\to H_-\bowtie H_+$
by 
\begin{equation*}
C\left((\wtd{B}(x),B(y))\right)
=\left(\wtd{B}\left(\wtd{B}(x_1)B(y_1)\right),S\circ B\circ S\left(\wtd{B}(x_2)B(y_2)\right)\right),
\end{equation*}
and $\wtd{C}:H_-\bowtie H_+\to H_-\bowtie H_+$ by 
\begin{equation*}
\wtd{C}\left((\wtd{B}(x),B(y))\right)
=\left(\wtd{B}\circ S\left(B\circ S(x_1)B(y_1)\right),S\circ B\left(\wtd{B}(x_2)\wtd{B}\circ S(y_2)\right) \right)
\end{equation*}
for any $x,y\in H.$
Then $(C,\wtd{C})$ is a projection homomorphism pair on $(H_+,H_-,\rhd,\ahd).$
\end{thm}
\begin{proof}
First, for any $x\in K_+$ and $y\in K_-,$ we have
$S(x)\in K_+$ and $S(y)\in K_-$. It follows that $\wtd{C}\left((\wtd{B}(x),B(y))\right)=0.$ This means that $\wtd{C}$ is well defined. Next, as $B$ and $\wtd{B}$ are both coalgebra maps, one can readily check that $C$ and $\wtd{C}$ are coalgebra maps. Then by \eqref{BTTB}, we know that $C$ and $\wtd{C}$ are idempotent. Next, we prove that $C$ is a homomorphism of Hopf algebras. For any $x,y,x',y'\in H$, we have
\begin{equation*}
\begin{aligned}
&C\left( (\wtd{B}(x),B(y))(\wtd{B}(x'),B(y'))\right)\\
=&C\left(\left(\wtd{B}(x_1) (B(y_1)\rhd \wtd{B}(x'_1)),(B(y_2)\ahd \wtd{B}(x'_2)) B(y')\right)\right)\\
=&\left(\wtd{B}\left(\wtd{B}(x_1)B(y_1)\wtd{B}(x'_1)B(y'_1)\right), S\circ B\circ S\left(\wtd{B}(x_2)B(y_2)\wtd{B}(x'_2)B(y'_2)\right)\right)\ (\textit{by \eqref{MM3}}).
\end{aligned}
\end{equation*}
On the other hand, we have
\begin{equation*}
\begin{aligned}
&C\left( (\wtd{B}(x),B(y))\right)C\left((\wtd{B}(x'),B(y'))\right)\\
=&\left(\wtd{B}\left(\wtd{B}(x_1)B(y_1)\right),S\circ B\circ S\left(\wtd{B}(x_2)B(y_2)\right)\right)\left(\wtd{B}\left(\wtd{B}(x'_1)B(y'_1)\right),S\circ B\circ S\left(\wtd{B}(x'_2)B(y'_2)\right)\right)\\
=&\left(\wtd{B}\left(\wtd{B}(x_1)B(y_1)\wtd{B}(x'_1)B(y'_1)\right), S\circ B\circ S\left(\wtd{B}(x_2)B(y_2)\wtd{B}(x'_2)B(y'_2)\right) \right)\ (\textit{by Proposition \ref{MRBE}}).
\end{aligned}
\end{equation*}
This shows that $C$ is a homomorphism of algebras.
As $B,\wtd{B}$ are both coalgebra maps, one can readily verify that $C$ is a coalgebra map. This proves that $C$ is a homomorphism of Hopf algebras. It follows from Proposition \ref{ATWB} that 
\begin{equation}\label{13}
\begin{aligned}
\wtd{C}\left(\wtd{B}(x),B(y)\right)
=&\left(\wtd{B}\circ S\left(B\circ S(y_1)B(x_1)\right),S\circ B\left(\wtd{B}(y_2)\wtd{B}\circ S(x_2)\right) \right)\\
=&\left(\wtd{B}\circ S\circ B(S(y_1)\ast_{B} x_1),S\circ B\circ \wtd{B}\left(y_2\ast_{\wtd{B}}S(x_2) \right) \right)\ (\textit{by Proposition \ref{AT}})\\
=&\left(\wtd{B}\circ S\circ B\circ S(y_1\ast_{\wtd{B}} S(x_1)),S\circ B\circ \wtd{B}\left(y_2\ast_{\wtd{B}}S(x_2) \right) \right)\ (\textit{by Proposition \ref{ATWB}}).\\
\end{aligned}
\end{equation}
Then by the proof of Lemma \ref{GY2} and \eqref{13}, we have
\begin{equation}\label{FE}
\begin{aligned}
C\left((\wtd{B}(x_1),B(y_1))\right)\wtd{C}\left(
(\wtd{B}(x_2),B(y_2))\right)=\wtd{C}\left(
(\wtd{B}(x_1),B(y_1))\right)C\left((\wtd{B}(x_2),B(y_2))\right)=(\wtd{B}(x),B(y)).
\end{aligned}
\end{equation}
Finally, it follows from Proposition \ref{CMM} that $(C,\wtd{C})$ is a projection homomorphism pair on $(H_+,H_-,\rhd,\ahd)$.
\end{proof}
By the definition of $C$ and $\wtd{C}$, one can check that 
$$
\begin{aligned}
\im C&= \left\{\left(\wtd{B}(x_1),S\circ B\circ S(x_2)\right)| x\in H  \right\}\\
\im\wtd{C}&=\left\{\left(\wtd{B}\circ S\circ B\circ S(x_1),S\circ B\circ \wtd{B}(x_2)\right)| x\in H  \right\}\\
\end{aligned}
$$

In the next theorem, we show that there is a Rota-Baxter Hopf algebra structure on the $\im C,$ which is Rota-Baxter isomorphic to $(H,B).$
\begin{thm}
Let $(H,B)$ be a Rota-Baxter Hopf algebra of weight $-1$ and $(H_+,H_-,\rhd,\ahd)$ be the matched pair of Hopf algebras on $(H,B).$ Let $(C,\wtd{C})$ be the projection homomorphism pair on $(H_+,H_-,\rhd,\ahd)$ and $H_1=\im C.$ Define $B_1:H_1\to H_1$ by
\begin{equation*}
B_1\left(x,y\right)
=C\left((x,\epsilon(y)1_H)\right),\ \forall (x,y)\in H_1,
\end{equation*}
and the operator $\wtd{B_1}:H_1\to H_1$ by
\begin{equation*}
\wtd{B_1}((x,y))=(x_1,y_1)B_1(S((x_2,y_2))),\ \forall (x,y)\in H_1.
\end{equation*}
Then $(H_1,B_1)$ is a Rota-Baxter Hopf algebra of weight $-1$ that is Rota-Baxter isomorphic to $(H,\wtd{B})$, and $(H_1,\wtd{B_1})$ is a Rota-Baxter Hopf algebra of weight $-1$ that is Rota-Baxter isomorphic to $(H,B).$
\end{thm}
\begin{proof}
First, it follows from Proposition \ref{WTB} and Theorem \ref{RBP} that $B_1$ and $\wtd{B_1}$ are both Rota-Baxter operators of weight $-1$ on $H.$
Define $\phi:H\to H_1$
by $$\phi(x)=\left(\wtd{B}(x_1),S\circ B\circ S(x_2)  \right),\ \forall x\in H.$$
Then we prove that $\phi$ is bijective. If $\phi(x)=\left(\wtd{B}(x_1),S\circ B\circ S(x_2)\right)=(0,0),$ then it follows from \eqref{BTTB} that 
$x=\wtd{B}(x_1)S\circ B\circ S(x_2)=0$. This means that $\phi$ is injective.
It follows directly from the definition that $\phi$ is surjective. Then we prove that $\phi$ is a homomorphism of Hopf algebras. By the definition, it is straightforward to see that $\phi$ is a coalgebra map. By Proposition \ref{MRBE}, we have
\begin{equation*}
\begin{aligned}
\phi(x)\phi(y)=&\left(\wtd{B}(x_1),S\circ B\circ S(x_2)  \right)\left(\wtd{B}(y_1),S\circ B\circ S(y_2)  \right)=\left(\wtd{B}(x_1y_1),S\circ B\circ S(x_2y_2)  \right)
\end{aligned}
\end{equation*}
for any $x,y\in H.$ This shows that $\phi$ is a homomorphism of Hopf algebras. Finally, we show that $\phi$ is compatible with $\wtd{B}$ and $B_1.$ We have
$$\phi\circ \wtd{B}(x)=\left(\wtd{B}(\wtd{B}(x_1)),S\circ B\circ S(\wtd{B}(x_2))\right)=
C\left((\wtd{B}(x),\epsilon(y)1_{H})\right)=B_1\circ\phi(x).
$$
This proves that $(H_1,B_1)$ is Rota-Baxter isomorphic to $(H,\wtd{B}).$ It is similar to prove that $(H_1,B_2)$ is Rota-Baxter isomorphic to $(H,B).$

\end{proof}

In the next theorem, we study the relationship between the Rota-Baxter Hopf algebra structure on $H_2$ and the descendent Rota-Baxter Hopf algebra $(H_B,B).$ 
\begin{thm}\label{FI}
Let $(H,B)$ be a Rota-Baxter Hopf algebra of weight $-1$ and $(H_+,H_-,\rhd,\ahd)$ be the matched pair of Hopf algebras on $(H,B).$ Let $(C,\wtd{C})$ be the projection homomorphism pair on $(H_+,H_-,\rhd,\ahd)$ and $H_2=\im \wtd{C}.$ Define the operators $B_2:H_2\to H_2$ and $\wtd{B_2}:H'\bowtie H\to H'\bowtie H$, $\pi:H\to H_2$ by
\begin{equation*}
\begin{aligned}
B_2\left(x,y\right)
&=\wtd{C}\left((x,\epsilon(y)1_H)\right),\ \forall (x,y)\in H_2,\\
\wtd{B_2}((x,y))&=(x_1,y_1)B_2(S((x_2,y_2))),\ \forall (x,y)\in H_2\\
\pi(x)&=\left(\wtd{B}\left(  S(B(x_1))\right),S(B(\wtd{B}(S(x_2))))  \right),\ \forall x\in H
\end{aligned}
\end{equation*}
respectively.
Then $\pi$ is a surjective homomorphism of Rota-Baxter Hopf algebras from $(H_B,B)$ to $(H_2,B_2)$, and $\pi\circ S$ is a surjective homomorphism of Rota-Baxter Hopf algebras from $(H_{\wtd{B}},\wtd{B})$ to $(H_2,
\wtd{B_2}).$
\end{thm}
\begin{proof}
By Proposition \ref{WTB} and Theorem \ref{RBP}, we obtain that $B_2$ and $\wtd{B_2}$ are both Rota-Baxter operators of weight $-1$ on $H.$ Then we prove that $\pi$ is a homomorphism of Hopf algebras. For any $x,y\in H,$ we have
\begin{equation*}
\begin{aligned}
\pi(x)\pi(y)=&\wtd{C}((\epsilon(x_1)1_{H},B(x_2)))\wtd{C}((\epsilon(y_1)1_{H},B(y_2)))
=\wtd{C}((\epsilon(x_1)1_{H},B(x_2))(\epsilon(y_1)1_{H},B(y_2)))\\
=&\wtd{C}((\epsilon(x_{1}y_{1})1_{H},B(x_2y_2))
=\left(\wtd{B}\circ S\left(B(x_1\ast_B y_1)\right),S\circ B\left(\wtd{B}\circ S(x_2\ast_B y_2)\right) \right)
=\pi(x\ast_{B} y).
\end{aligned}
\end{equation*}
This proves that $\pi$ is a homomorphism of Hopf algebras. Next, we show that $\pi$ is surjective. For any $x,y\in H,$ we have 
$$
\begin{aligned}
\pi(S(x)\ast_B y)=&\left(  
\wtd{B}\left(S(B(S(x_1)\ast_B y_1))\right),S\left(B(\wtd{B}(S(S(x_2)\ast_B y_2)))\right)
\right)\\
=&\left(  
\wtd{B}(S(B(S(x_1)\ast_B y_1))),S\left(B(\wtd{B}(x_2\ast_{\wtd{B}} S(y_2)))\right)
\right)\ (\textit{by Proposition \ref{ATWB}})\\
=&\left(  
\wtd{B}\left(S(B(S(x_1)) B(y_1))\right),S\left(B(\wtd{B}(x_2)\wtd{B}(S(y_2)))\right)
\right)\ (\textit{by Proposition \ref{AT}})\\
=&\wtd{C}\left((\wtd{B}(x),B(y))\right).
\end{aligned}
$$
This shows that $\pi$ is surjective. Finally, we show that $\pi$ is compatible with $B$ and $B_2.$ We have
$$
\begin{aligned}
\pi\circ B(x)=&\left(\wtd{B}\left(  S(B(B(x_1)))\right),S\left(B\left(\wtd{B}(S(B(x_2)))\right)\right)\right)
=\wtd{C}\left(\wtd{B}(S(B(x_1))),\epsilon(x_2)\right)\\=&B_2\left(\wtd{B}\left(  S(B(x_1))\right),S\left(B\left(\wtd{B}(S(x_2))\right)\right)\right)=B_2(\pi(x)).
\end{aligned}
$$
Therefore, $\pi$ is a homomorphism of Rota-Baxter Hopf algebras from $(H_B,B)$ to $(H_2,B_2)$. It is similar to prove that $\pi$ is a surjective homomorphism of Rota-Baxter Hopf algebras from $(H_{\wtd{B}},\wtd{B})$ to $(H_2,
\wtd{B_2}).$
\end{proof}

\vspace{2mm}
\noindent
{\bf Acknowledgements. } This research is supported by Scientific Research Foundation for High-level Talents of Anhui University of Science and Technology (2024yjrc49).

\smallskip

\vspace{2mm}
\noindent
{\bf Declaration of interests. } The authors have no conflicts of interest to disclose.

\smallskip

\vspace{2mm}
\noindent
{\bf Data availability. } Data sharing not applicable to this article as no datasets were generated or analysed in this work.

\vspace{-.2cm}

%\end{proof}

\end{document}